\documentclass[12pt]{amsart}
\usepackage{amssymb,mathrsfs,longtable}
\usepackage[matrix,arrow,curve]{xy}
\usepackage{setspace}
\usepackage{multirow}
\usepackage{textcomp}
\usepackage{hyperref}

\usepackage[dvipsnames]{xcolor}
\usepackage{graphicx}

\usepackage[normalem]{ulem}
\usepackage{cancel}

\newcounter{cequation}

\newtheorem{theorem}%[cequation]
{Theorem}
\newtheorem*{theorem*}{Theorem}
\newtheorem{lemma}%[cequation]
{Lemma}
\newtheorem{corollary}%[cequation]
{Corollary}
{Conjecture}
\newtheorem{proposition}%[cequation]
{Proposition}

\theoremstyle{definition}
\newtheorem{example}%[cequation]
{Example}
\newtheorem{definition}%[cequation]
{Definition}
\newtheorem*{definition*}{Definition}

\newtheorem*{notation*}{Notation}

\theoremstyle{remark}
\newtheorem{remark}[cequation]{Remark}

\makeatletter\@addtoreset{equation}{section}
\makeatletter\@addtoreset{section}{part}

\makeatother

\def \CC {\mathbb{C}}

\def \PP {\mathbb{P}}
\def \AA {\mathbb{A}}
\def \QQ {\mathbb{Q}}
\def \ZZ {\mathbb{Z}}
\def \Cl {\mathrm{Cl}}
\def \Pic {\mathrm{Pic}}

\def \ge {\geqslant}

\title{Smooth prime Fano complete intersections in toric varieties}

\author{Victor Przyjalkowski and Constantin Shramov}

\address{\emph{Victor Przyjalkowski}
\newline
\textnormal{Steklov Mathematical Institute of RAS, 8 Gubkina street, Moscow 119991, Russia.
}
\newline
\textnormal{\texttt{victorprz@mi.ras.ru, victorprz@gmail.com}}}

\address{\emph{Constantin Shramov}
\newline
\textnormal{Steklov Mathematical Institute of RAS,
8 Gubkina street, Moscow 119991, Russia.
}
\newline
\textnormal{National Research University Higher School of Economics, Laboratory of Algebraic Geometry, NRU HSE, 6 Usacheva str., Moscow, 117312, Russia.
}
\newline
\textnormal{\texttt{costya.shramov@gmail.com}}}

\thanks{Victor Przyjalkowski was supported by
President Grant MD-30.2020.1. Constantin Shramov was supported by the
HSE University Basic Research Program,
Russian Academic Excellence Project~\mbox{``5-100''}.
Both authors are supported by the Foundation for the
Advancement of Theoretical Physics and Mathematics ``BASIS''}

\begin{document}

\begin{abstract}
We prove that a smooth well formed Picard rank one Fano complete intersection of dimension at least $2$ in a toric variety is a weighted complete intersection.
\end{abstract}

\maketitle

An obvious way to construct new Fano varieties
is to describe them as divisors or complete intersections of divisors
in those varieties that are already
known and well understood:
for instance, in toric varieties (in particular, in weighted projective spaces), or in Grassmannians.
Fano varieties that are complete intersections in Grassmannians
were studied by several authors, see, for instance,~\cite{Kuchle-Grass},~\cite{Kuchle-Geography}, or~\cite{FM}.
As for Fano weighted complete intersections, they also attracted some attention, see e.g.~\cite{ChenChenChen},~\cite{PST},~\cite{PrzyalkowskiShramov-AutWCI},~\cite{PrzyalkowskiShramov-Weighted},~\cite{PrzyalkowskiShramov-Siberia},~\cite{PrzyalkowskiShramov-Albania}.
The most interesting case from the point of view of classification
is the case of prime Fano varieties, that is ones of Picard rank one. Note that this property automatically
holds for smooth complete intersections (of dimension at least~$3$)
in Grassmannians or weighted projective spaces by the Lefschetz theorem.
On the other hand, for complete intersections in other toric
varieties it does not hold often.
In this note we prove that the only prime Fano complete intersections of dimension at least $2$ in toric varieties
are weighted complete intersections.

First we remind the classical construction of a toric variety as a quotient (see~\cite{Co95}).
Let $Y$ be a $\QQ$-factorial projective toric variety; here and below all varieties are assumed to be defined
over the field $\mathbb{C}$ of complex numbers. Let~\mbox{$D_1,\ldots, D_b$} be its prime boundary divisors.
They generate the divisor class group~\mbox{$\Cl(Y)$}.
Consider the $\Cl(Y)$-graded algebra~\mbox{$S=\CC[x_1,\ldots, x_b]$} with grading defined by
$$
\deg\left(\prod_{i=1}^b x_i^{r_i}\right)=\sum_{i=1}^b r_i D_i.
$$
One has $\mathrm{Spec}\,(S)\cong\AA^b$, and there is a natural correspondence between rays $e_i$ of a fan of $Y$ and variables $x_i$.
Define the subvariety $Z$ in $\mathrm{Spec}\,(S)$ as the intersection of hypersurfaces~\mbox{$\{\prod x_i=0\mid \ e_i\notin \sigma\}$} over all cones~$\sigma$ of a fan of~$Y$. Then $Y$ is a geometric quotient
of 
$$
U=\mathrm{Spec}\,(S)\setminus Z\subset\AA^b
$$
by the abelian linear algebraic group
$$
\mathbf{D}=\mathrm{Hom}_\ZZ(\Cl(Y),\CC^*).
$$

\begin{example}
\label{example:wps grading}
Let $Y=\PP(a_0,\ldots,a_N)$ be a weighted projective space of dimension $N$.
Suppose that the greatest common divisor of the numbers
$a_0,\ldots,a_N$ equals~$1$ (which can be always achieved by cancelling
the greatest common divisor of~\mbox{$a_0,\ldots,a_N$}).
Recall that nevertheless the collection $a_0,\ldots,a_N$
is \emph{not} uniquely defined by $Y$. However for a suitable choice of the weights
$a_0,\ldots,a_N$ the fan of the toric variety $Y$ in $\ZZ^N$
has $N+1$ rays with primitive vectors $v_0,\ldots,v_N$ subject to
the relation
\begin{equation}\label{eq:relation}
a_0v_0+\ldots+a_Nv_N=0,
\end{equation}
see~\cite[1.2.5]{Do82}.
Using the description of $\Cl(Y)$ via the rays of the fan, we see that~\mbox{$\Cl(Y)\cong \ZZ$}.
This means that $\mathbf{D}\cong\CC^*$.
Furthermore, one has~\mbox{$Z=\{0\}\subset \AA^{N+1}$},
so that
$$
Y\cong\left(\AA^{N+1}\setminus \{0\}\right)/\CC^*.
$$
\end{example}

\begin{example}[{\cite[Example 1.3]{Kasprzyk}}]
Let $Y$ be a three-dimensional toric variety
whose fan has rays generated by the vectors
$$(1,0,0),\ \ (0,1,0),\ \
(1,-3,5),\ \ (-2,2,-5)
$$
in $\ZZ^3$.
One can see that
$$
\Cl(Y)\cong \ZZ\oplus \ZZ/5\ZZ.
$$
Therefore, we have
$$
\mathbf{D}=\mathrm{Hom}_\ZZ(\ZZ\oplus \ZZ/5\ZZ,\CC^*)\cong\CC^*\times \ZZ/5\ZZ.
$$
In particular, $\mathbf{D}$ is not connected in this case.
One has
$$
Y\cong \left(\AA^4\setminus\{0\}\right)/\mathbf{D}
\cong\PP^3/\left(\ZZ/5\ZZ\right).
$$
\end{example}

\begin{remark}
According to \cite{BorisovBorisov}, the threefold $Y$ is the unique
three-dimensional toric variety with terminal singularities that
has Picard rank $1$ and is not a weighted projective space.
\end{remark}

\begin{example}
\label{example:general type 3fold}
Let $Y$ be a four-dimensional toric variety
whose fan has rays generated by the vectors
$$
(1,0,0,0),\ \ (0,1,0,-1),\ \ (0,0,-1,0),\ \ (0,0,2,-1),\ \ (-1,-1,-1,2)
$$
in $\mathcal N\cong \ZZ^4$.
The sum of these vectors is equal to $0$, so $Y$ is a
quotient of $\PP^4$. In fact, one has
$$
Y\cong \left(\AA^4\setminus\{0\}\right)/\left(\CC^*\times \ZZ/3\ZZ\right)
\cong\PP^4/\left(\ZZ/3\ZZ\right).
$$
Since any three of the above five vectors form
a part of a basis of $\mathcal N$,
we see that the singularities of $Y$ are isolated.
\end{example}

\begin{example}[{cf.~\cite[Example 1.2]{Kasprzyk}}]
\label{example:general general type}
Let $p$ be a positive integer, and let~\mbox{$G=\ZZ/p\ZZ$}.
Consider the action of $G$ on the projective space $\PP^{p-1}$
with homogeneous coordinates~\mbox{$x_0,\ldots,x_{p-1}$}
given by the formula $x_i\mapsto \varepsilon^i x_i$,
where $\varepsilon$ is a primitive $p$-th root of unity.
Set~\mbox{$Y=\PP^{p-1}/G$}.
Then $Y$ is a $\QQ$-factorial toric variety,
and
$$
Y\cong \left(\AA^{p}\setminus\{0\}\right)/\left(\CC^*\times \ZZ/p\ZZ\right).
$$
Note that if $p$ is a prime number, then the singularities of $Y$
are isolated.
\end{example}

The above examples motivate the following definition.

\begin{definition}[{cf.~\cite[Definition 1.1]{Kasprzyk}}]
\label{definition:FPS}
An $N$-dimensional toric variety whose fan has~\mbox{$N+1$} rays is called a \emph{generalized weighted projective space}.
\end{definition}

\begin{lemma}[{\cite[Corollary 2.3]{Kasprzyk}}]
\label{lemma:quotient Pic 1}
Every generalized weighted projective space is either a weighted projective space,
or a quotient
of a weighted projective space by a non-trivial finite group  acting freely in codimension~$1$.
\end{lemma}

We call a polynomial~\mbox{$f\in S$} homogeneous if for some~\mbox{$d\in \Cl(Y)$} all the monomials of $f$ are of degree $d$.
For any homogeneous polynomials $f_1,\ldots, f_k$
the intersection $C_X^*$ of their common zero locus with $U$ is stable under the action of $\mathbf{D}$.
Thus $C_X^*$ determines a closed subset $X$ in~$Y$.

Note that this description specializes to weighted complete intersections in the case when $Y$ is a weighted projective space.
In the same way as for weighted complete intersections we can give standard definitions
for the case of complete intersections in toric varieties.
Let $X\subset Y$ have codimension~$k$, so that~\mbox{$C_X^*\subset U$} is also of codimension $k$.
If there are $k$ generators $f_1,\ldots, f_k$ of the ideal of~$C_X^*$ in~$U$, so that $C_X^*$ is a complete intersection in $U$,
we say that $X$  is a \emph{complete intersection} of the hypersurfaces that are images in $Y$ of the divisors defined by equations $f_j=0$ in~$U$.
This is equivalent to
the regularity of the sequence $f_1,\ldots, f_k$ in the localization of $S$ with respect to the ideal defining $Z$. We say that $X$ is \emph{well formed}
if
$$
\mathrm{codim}_X \left( X\cap\mathrm{Sing}\,Y \right)\ge 2.
$$

\begin{remark}
Let $Y$ be a weighted projective space $\PP(a_0,\ldots,a_N)$
considered as a toric variety.
We may assume that the greatest common divisor of the numbers
$a_0,\ldots,a_N$ equals~$1$.
According to Example~\ref{example:wps grading}, for a suitable choice of the weights
$a_0,\ldots,a_N$ the fan of the toric variety $Y$ in $\ZZ^N$
has $N+1$ rays with primitive vectors $v_0,\ldots,v_N$ subject to
relation~\eqref{eq:relation}.
If all numbers among $a_0,\ldots,a_N$ except one, say, the numbers
$a_1,\ldots,a_N$, are divisible by some integer $a>1$, then
all coordinates of the vector $v_0$ are divisible by~$a$; the latter is
impossible because $v_0$ is a primitive vector in~$\ZZ^N$.
Thus $Y$ is automatically well formed
in the sense of~\cite[Definition 5.11]{IF00}.
In other words, the description of $Y$ as a toric variety recovers the
unique collection $a_0,\ldots,a_N$
among all collections of weights defining the same weighted
projective space $Y$ such that the weighted projective space is
well formed (recall that
well formedness in the sense of~\cite{IF00} is a property of the weights, not the
weighted projective space itself). As a consequence, we conclude that
the two notions of well formedness for weighted complete intersections
agree with each other.
\end{remark}

It appears that many cohomology groups of a complete intersection
of ample hypersurfaces in a toric variety $Y$
are defined by the cohomology groups of $Y$,
similarly to complete intersections in usual projective spaces.

\begin{theorem}[{Lefschetz theorem for toric varieties, see~\cite[Proposition~1.4]{Ma99}}]
\label{theorem:Lefschetz}
Let~$Y$ be a $\mathbb{Q}$-factorial projective toric variety of dimension $N$, and let $X\subset Y$ be a complete intersection of $k$ ample hypersurfaces in $Y$.
Then the natural map
$$
H^i(Y,\ZZ)\to H^i(X,\ZZ)
$$
is an isomorphism for~\mbox{$i<N-k=\dim(X)$} and an injection
for $i=N-k$.
\end{theorem}

\begin{corollary}
\label{corollary:connectedness}
A positive-dimensional complete intersection of ample hypersurfaces in a $\mathbb{Q}$-factorial projective toric variety is connected.
\end{corollary}

\begin{corollary}
\label{corollary:Pic}
Let $X$ be a complete
intersection of ample hypersurfaces
in a $\mathbb{Q}$-factorial projective toric variety $Y$.
Then the restriction map
$$
\Pic(Y)\to\Pic(X)
$$
is an isomorphism if $\dim(X)\ge 3$, and
is injective if $\dim(X)=2$.
Furthermore, if~\mbox{$\dim(X)\ge 2$} and
$X$ is $\mathbb{Q}$-factorial, then  $\mathrm{rk\,Cl}(X)\ge\mathrm{rk\,Cl}(Y)$.
\end{corollary}

\begin{proof}
Note that $H^1(Y,\ZZ)=0$, because $Y$ is a toric variety.
By Theorem~\ref{theorem:Lefschetz}, one has~\mbox{$H^1(X,\ZZ)=0$};
moreover, the map $H^2(X,\ZZ)\to H^2(Y,\ZZ)$
is an isomorphism if~\mbox{$\dim(X)\ge 3$}, and
is injective if $\dim(X)=2$.
Thus, the assertions about the Picard group follow
from the exponential exact sequence.
On the other hand,
if $X$ is $\mathbb{Q}$-factorial,
one has $\mathrm{rk\,Cl}(X)=\mathrm{rk\,Pic}(X)$ (and also
$\mathrm{rk\,Cl}(Y)=\mathrm{rk\,Pic}(Y)$).
\end{proof}

For the following we need a general fact (see, for instance,~\cite[Corollary~3]{DD85}).

\begin{proposition}
\label{proposition:branching}
Let $U$ and $V$ be normal varieties, and let $f\colon U\to V$ be
a finite surjective
morphism.
Then the branch locus $B(f)$ of $f$ has codimension $1$
at any point of~\mbox{$B(f)\cap \left(V\setminus \mathrm{Sing}\, (V)\right)$}.
\end{proposition}

\begin{proposition}\label{proposition:quotient}
Let $X$ be a positive-dimensional well formed Fano complete intersection in a $\mathbb{Q}$-factorial projective toric variety $Y$.
Suppose that $Y$ is a quotient
of a weighted projective space by a non-trivial finite group $G$ acting freely in codimension~$1$.
Then $X$ is singular.
\end{proposition}

\begin{proof}
Suppose that $X$ is smooth.
Let $\psi\colon \PP\to Y$ be the quotient map, where $\PP$ is a weighted projective space.
We claim that the branch locus $B(\psi)\subset Y$ of $\psi$ lies in the singular locus of $Y$.
Indeed, by Proposition~\ref{proposition:branching} the locus $B(\psi)\setminus \mathrm{Sing}\,Y$ has codimension $1$ in $Y$.
Together with the freeness in codimension $1$ of the action of $G$ this implies that
$$
B(\psi)\setminus \mathrm{Sing}\,Y=\varnothing.
$$
Proposition~\ref{proposition:branching} applied to the covering $X'\to X$, where $X'=\psi^{-1}(X)$, implies that this covering is branched in the locus $B(\psi)\cap X$ of codimension $1$ in $X$.

Suppose that  $B(\psi)\cap X=\varnothing$, so that $X'\to X$ is an unramified covering. Note that~\mbox{$X'$} is a complete intersection in $Y$;
since $Y$ is a weighted projective space,~\mbox{$X'$} is a complete intersection of ample hypersurfaces.
Thus it follows from Corollary~\ref{corollary:connectedness}
that~\mbox{$X'$} is connected. On the other hand, the smooth Fano variety $X$ has trivial fundamental group (see, for instance,~\cite[Corollary~6.2.18]{IP99}). The obtained contradiction shows that~\mbox{$B(\psi)\cap X\neq\varnothing$}.
However, since
$$
B(\psi)\cap X\subset \mathrm{Sing}\,Y\cap X,
$$
the latter contradicts the well formedness
of $X$.
\end{proof}

\begin{theorem}\label{theorem:picone is usual}
Let $X$ be a smooth well formed Fano
complete intersection of ample hypersurfaces in a $\mathbb{Q}$-factorial projective toric variety~$Y$.
Suppose that~\mbox{$\dim(X)\ge 2$} and~\mbox{$\mathrm{rk\,Pic}(X)=1$}.
Then $Y$ is a weighted projective space.
\end{theorem}

\begin{proof}
We know from Corollary~\ref{corollary:Pic} that~\mbox{$\mathrm{rk\,Cl}(Y)=1$}.
In other words, $Y$ is a generalized weighted projective space.
Thus by Lemma~\ref{lemma:quotient Pic 1}
the variety~$Y$ is either a weighted projective space,
or a quotient
of a weighted projective space by a non-trivial finite group  acting freely in codimension~$1$.
Now the assertion follows from Proposition~\ref{proposition:quotient}.
\end{proof}

Note that the assertion of Theorem~\ref{theorem:picone is usual} fails if $X$ has dimension $1$.
For instance, a smooth rational curve can be embedded into the toric surface $Y=\PP^1\times\PP^1$ as an ample hypersurface
of bidegree $(1,1)$.
Also, Theorem~\ref{theorem:picone is usual} does not hold without the assumption that
$X$ is Fano. Indeed, let $Y$ be a generalized weighted projective space
of dimension at least $4$ with isolated singularities
such that $Y$ is not a weighted projective space, see
Examples~\ref{example:general type 3fold}
and~\ref{example:general general type}.
Let $X$ be a general divisor from a sufficiently ample
linear system on~$Y$. Then $X$ is smooth by Bertini theorem,
and~\mbox{$\mathrm{rk\,Pic}(X)=1$} by Corollary~\ref{corollary:Pic}.
Finally, we point out that the smoothness assumption in
Theorem~\ref{theorem:picone is usual} is also essential.
Indeed, let~\mbox{$p\ge 5$} be a prime number, and let $H$
be the hyperplane defined by the equation~\mbox{$x_{p-1}=0$} in
the projective space $\PP^{p-1}$ with homogeneous coordinates
$x_0,\ldots,x_{p-1}$. Let~$Y$ be the generalized weighted projective space constructed as
in Example~\ref{example:general general type}. Then the image of $H$ in $Y$ is a prime Fano hypersurface with terminal
singularities.

\medskip
The authors are grateful to Yu.\,Prokhorov for his helpful comments.

\end{document}